\documentclass[letterpaper]{amsart}
\usepackage{amsaddr}
\makeatletter
\renewcommand{\email}[2][]{%
  \ifx\emails\@empty\relax\else{\g@addto@macro\emails{,\space}}\fi%
  \@ifnotempty{#1}{\g@addto@macro\emails{\textrm{(#1)}\space}}%
  \g@addto@macro\emails{#2}%
}
\makeatother

\usepackage{amsxtra}
\usepackage{extarrows}
\usepackage{color}
\usepackage{xcolor}
\usepackage{enumitem}
\usepackage{subcaption}
\setlist[enumerate]{label=\upshape(\arabic*)}
\allowdisplaybreaks[4]

\usepackage{upref}

\usepackage{mathtools}
\usepackage{hyperref}
\usepackage{amsrefs}
\usepackage{amsthm}
\usepackage{amssymb}
\usepackage{amsmath}
\numberwithin{equation}{section}

\newtheorem{thm}{Theorem}
\newtheorem{lem}[thm]{Lemma}
\newtheorem{prop}[thm]{Proposition}
\newtheorem{cor}[thm]{Corollary}
\newtheorem{conj}[thm]{Conjecture}
\newtheorem{ques}[thm]{Question}

\theoremstyle{definition}

\theoremstyle{remark}
\newtheorem{rmk}[thm]{Remark}

\begin{document}

\title{Equivariant log-concavity and equivariant K\"ahler packages}

\author{Tao Gui}
\address[A1]{Academy of Mathematics and Systems Science, Chinese Academy of Sciences, Beijing, P.R. China}
\email[]{guitao18@mails.ucas.ac.cn}

\author{Rui Xiong}
\address[A2]{Department of Mathematics and Statistics, University of Ottawa, 150 Louis-Pasteur, Ottawa, ON,
K1N 6N5, Canada}
\email[]{rxion043@uottawa.ca}

\maketitle

\begin{abstract}
We show that the exterior algebra $\Lambda_{R}\left[\alpha_{1}, \cdots, \alpha_{n}\right]$, which is the cohomology of the torus $T=(S^{1})^{n}$, and the polynomial ring $\mathbb{R}\left[t_{1}, \ldots, t_{n}\right]$, which is the cohomology of the classifying space $B (S^{1})^{n}=\left(\mathbb{C} \mathbb{P}^{\infty}\right)^{n}$, are $S_{n}$-equivariantly log-concave. We do so by explicitly giving the $S_{n}$-representation maps on the appropriate
sequences of tensor products of polynomials or exterior powers and proving that these maps satisfy the hard Lefschetz theorem. Furthermore, we prove that the whole K\"ahler package, including the Poincar\'e duality, the hard Lefschetz theorem, and the Hodge--Riemann bilinear relations, holds on the corresponding sequences in an equivariant setting.

\end{abstract}

\section{Introduction}
 
The \emph{K\"ahler package} is an algebraic structure that satisfies a “package” of properties tending to come bundled together: Poincar\'e duality, hard Lefschetz,
and Hodge--Riemann bilinear relations. They were first discovered in the cohomology of smooth complex projective algebraic varieties, or more generally, compact K\"ahler manifolds, but have since become quite ubiquitous in a realm that goes beyond that of K\"ahler geometry. For example, see \cites{mcmullen1993simple,barthel2002combinatorial,bressler2003intersection,karu2004hard,braden2006remarks,elias2014hodge,zbMATH06786900,zbMATH07367669,adiprasito2018hodge,braden2020singular}.
In this paper, we discover two new K\"ahler packages that are equivariant and have no geometric origin, see Theorem \ref{main} (Theorem \ref{kah} and \ref{kah2}). The equivariant log-concavity hints at our discoveries for these structures.

Recall that a sequence of real numbers $a_{0}, a_{1}, \ldots$ is called \emph{log-concave} if 
$$
a_{i}^{2} \geq a_{i-1} a_{i+1} \text { for all $i\geq 1$} .
$$
A finite log-concave sequence  of positive numbers is always \emph{unimodal}, that is, there exists an index $i$ such that
$$a_{0} \leq \cdots \leq a_{i-1} \leq a_{i} \geq a_{i+1} \geq \cdots \geq a_d.$$
Log-concave and unimodal sequences turn out to be very common in algebra, geometry, and combinatorics, see \cites{stanley1989log, brenti1994combinatorics,stanley2000positivity}. 
However, proving that certain specific sequences are log-concave might be quite difficult. Some recent work \cites{huh2012milnor, adiprasito2018hodge, branden2020lorentzian} found a deep connection between log-concavity and Hodge--Riemann bilinear relations of Hodge theory, see the ICM paper  \cite{huh2018combinatorial} for a survey.

Many log-concave phenomena appear in representation theory and algebraic topology.
For example, Nakajima \cite{nakajima2003t} proved that the Kirillov--Reshetikhin modules satisfy log-concave-type properties known as the $T$-system and $Q$-system for quantum loop algebras. Okounkov \cite{okounkov1996brunn} introduced the Newton--Okounkov bodies to prove that, in the asymptotic sense, the multiplicities of the homogeneous coordinate ring of a $G$-stable irreducible subvariety of $\mathbf{P}(V)$ are log-concave for any representation $V$ of a reductive group $G$. 
In a sequel paper \cite{okounkov2003would}, based on heuristics and analogies of physical principles, Okounkov made a remarkable conjecture that the Littlewood--Richardson coefficients $c_{\lambda\mu}^{\nu}$ are log-concave in $(\lambda, \mu, \nu)$. Although Okounkov's conjecture is false in general \cite{chindris2007counterexamples}, many consequences and interesting special cases are true, see for example \cites{knutson1999honeycomb,lam2007schur,huh2022logarithmic}.

Another log-concavity phenomenon comes from group actions. Following \cite{matherneequivariant}, see also \cite{okounkov2003would}, we say that a sequence of finite-dimensional representations $V_{0}, V_{1}, \ldots$, for example, a graded representation $V=\bigoplus_{i=0}^\infty V_i$, is
\emph{equivariantly log-concave} 
if 
\begin{equation} \label{equ}
V_{i-1} \otimes V_{i+1} \text { is a subrepresentation of } V_{i} \otimes V_{i}, \text { for all } i\geq 1, 
\end{equation} 
which is a categorification of the log-concavity of its dimension sequence. 
Suppose a group $G$ acts on a topological space $X$, which induces a graded representation of $G$ on the cohomology $H^{*}(X)$. If $H^{*}(X)$ is equivariantly log-concave, then the Betti numbers of $X$ form a log-concave sequence. In \cite{matherneequivariant}, the authors made $G$-equivariant log-concavity conjectures for the cohomology of various configuration spaces and exploited the theory of representation stability to give computer-assisted proofs in low degrees for $G=S_{n}$. But despite the strong experimental evidence, they did not propose any strategy that could prove $S_{n}$-equivariant log-concavity for all $n$ in all degrees, and the full conjectures remain open.

Our results can be viewed as an attempt to attack these equivariantly log-concave conjectures. Our approach is to reduce the problem of equivariant log-concavity to the equivariant unimodality of appropriate tensor product sequences and establish the hard Lefschetz theorem to prove the unimodality. 
In this paper, we use this idea to establish the equivariant log-concavity for the polynomial ring and the exterior algebras by explicit construction of equivariant Lefschetz operators. Actually, more is true ---the Hodge--Riemann relations also hold which forms an \emph{equivariant K\"ahler package} combined with Poincar\'e duality and hard Lefschetz. 

Denote 
$$H_{n,m}=\bigoplus_{i = 0}^{m} H_{n,m}^{-m+2i}, \text{ with }H_{n,m}^{-m+2i}:=D^{i} \otimes R^{m-i},$$
where $D^{i}$ (resp., $R^{m-i}$) denotes the degree-$i$ (resp., degree-($m-i$)) piece of the polynomial ring $D=\mathbb{R}\left[d_{1}, \ldots, d_{n}\right]$ (resp., $R=\mathbb{R}\left[x_{1}, \ldots, x_{n}\right]$). Let $$L: H_{n,m}^{i} \longrightarrow H_{n,m}^{i+2}, \text{ where }
L:=\sum_{i=1}^{n} d_{i} \otimes \frac{\partial}{\partial x_{i}}.
$$ We define a pairing on $H_{n,m}$, see formula \eqref{pair}.

Similarly, denote
$$H_{n,m}^{\prime}=\bigoplus_{i=0}^{m} (H_{n,m}^{\prime})^{-m+2 i}, \text { with } (H_{n,m}^{\prime})^{-m+2 i}:=\Lambda^{i} \otimes (\Lambda^{*})^{m-i},$$
where $\Lambda^{i}$ (resp., $(\Lambda^{*})^{m-i}$) denotes the degree- $i$ (resp., degree-($m-i$)) piece of the exterior algebra $\Lambda(V)=\Lambda_{\mathbb{R}}\left[\theta_{1}, \ldots, \theta_{n}\right]$ (resp., $\Lambda(V^{*})=\Lambda_{\mathbb{R}}\left[\xi_{1}, \ldots, \xi_{n}\right]$), we demand that $\theta_{1}, \ldots, \theta_{n}$ and $\xi_{1}, \ldots, \xi_{n}$ form dual basis of $V$ and $V^{*}$.
Let $$L: (H_{n,m}^{\prime})^{i} \longrightarrow (H_{n,m}^{\prime})^{i+2}, \text{  where } L:=\sum_{k=1}^{n} e_{\theta_{k}} \otimes i_{\theta_{k}},$$ here $e_{\theta_{k}}$ (resp., $i_{\theta_{k}}$) is the exterior (resp., interior) product with $\theta_{k}$. We define a pairing on $H_{n,m}^{\prime}$, see formula \eqref{pair2}. 

$S_{n}$ naturally acts on the spaces. The followings are our main results.

\begin{thm}[Theorem \ref{kah} and \ref{kah2}] \label{main}
Under the notations introduced above, for $H=H_{n,m}$ or $H_{n,m}^{\prime}$, we have
\begin{itemize}
\item[(a)] (Poincar\'e duality)  For all $i \geq 0$, the bilinear pairing 
$$
\langle-,-\rangle: H^{i} \times H^{-i} \longrightarrow \mathbb{R}
$$ 
is $S_{n}$-equivariant and non-degenerate;
\item[(b)] (Hard Lefschetz) The corresponding operator $L: H^{i} \longrightarrow H^{i+2}$ is $S_{n}$-equivariant and satisfying hard Lefschetz, i.e., for all $i \geq 0$,
\begin{equation*}
L^{i}: H^{-i} \longrightarrow H^{i}
\end{equation*}
is an isomorphism.
\item[(c)] (Hodge--Riemann relations) For all $0 \leq i \leq m/2$, the Lefschetz form on $H^{-m+2i}$, defined by
$$(a, b)_{L}^{-m+2i}:=\left\langle a, L^{m-2i} b\right\rangle,$$
is $S_{n}$-equivariant and $(-1)^{i}$-definite on $P_{L}^{-m+2i}:=\ker\left(L^{m-2i+1}\right) \cap H^{-m+2i}$.
\end{itemize}
\end{thm}

The corresponding equivariant Lefschetz operators give the explicit maps realizing the equivariant log-concavity of the symmetric powers and the exterior powers as graded representations of $S_{n}$. And, in fact, they are equivariant for the symmetric powers and the exterior powers of any representation of any group $G$; see Remark \ref{wider}.

\begin{cor}
The symmetric powers (polynomial ring) and the exterior powers (exterior algebra) of a given representation $V$ of any group $G$ are equivariantly log-concave. Furthermore, the required injective maps are given by the corresponding equivariant Lefschetz operators $L$ above.
\end{cor}

For polynomial rings, the Lefschetz operator $L$ is essentially the same as observed in \cite[Section 3.1]{haiman1994conjectures}, but we rediscover it from a different perspective: $d_{1}, \ldots, d_{n}$ and $x_ {1}, \ldots, x_{n}$ should live in the dual vector spaces and form a dual basis. Thus, our pairing \eqref{pair} differs from the more classical \emph{apolar form} in diagonal harmonic theory. Our pairing \eqref{pair} is compatible with the Lefschetz operator $L$, and we discover that the resulting Lefschetz form satisfies Hodge--Riemann bilinear relations, which is much stronger than hard Lefschetz, forming a full K\"ahler package combined with Poincar\'e duality and hard Lefschetz.

For the exterior algebra, the Lefschetz operator $L$ is new, and we establish the full K\"ahler package in the equivariant setting. We also find that the “usual” gradings on tensor products of the exterior algebra satisfy Poincar\'e duality and hard Lefschetz but not Hodge--Riemann relations, see Theorem \ref{kah3} and Remark \ref{coun}. The hard Lefschetz theorem for this grading is essentially the same as \cite[Theorem 3.2]{kim2022lefschetz}, but our proof is simpler than theirs since their proof is combinatorial and relies on the incidence matrix of the Boolean poset. They use the equivariant Lefschetz element as a key tool to study the \emph{fermionic diagonal coinvariants}. We would like to see whether our constructions have other implications for the diagonal coinvariant ring and the fermionic diagonal coinvariant ring.

The detailed constructions and proofs of the aforementioned theorems will be given in Section \ref{section:eqKah}. 
Our constructions have the following notable features:
\begin{itemize}
\item The Poincar\'e pairing $\langle-,-\rangle$, the Lefschetz operator $L$, and the Lefschetz form $(a, b)_{L}^{-m+2i}$ are all $S_{n}$-equivariant, which is rare in Lefschetz theory.
\item The adjoint operator of the Lefschetz operator $L$ can be written explicitly. 
\item The proof of hard Lefschetz and Hodge--Riemann relations is natural and it takes advantage of the geometry of the product of projective spaces. 
\end{itemize}

It would be interesting to know whether our approach could shed some new light on other (equivariant) log-concavity questions and conjectures.

\setcounter{tocdepth}{2}

\section{Preliminaries: K\"ahler package and Lefschetz linear algebra}

In this section, we recall properties in the K\"ahler package in the linear algebra context and refer to  \cite[Section 2]{elias2014hodge} and \cite[Chapter 17]{elias2020introduction} for the omitted details.

Consider a graded finite-dimensional real vector space
$$H=\bigoplus_{i \in \mathbb{Z}} H^{i}.$$
We denote $b_{i}=\operatorname{dim} H^{i}$, the $i$-th \emph{Betti number} of $H$. We say $H$ satisfies \emph{Poincar\'e duality} (PD) if there exists a non-degenerate symmetric graded bilinear form
$$\langle-,-\rangle: H \times H \longrightarrow \mathbb{R} ,$$
where \emph{graded} means that $\left\langle H^{i}, H^{j}\right\rangle=0$ if $i \neq-j$.
It follows that $\langle-,-\rangle$ induces an isomorphism between $H^{-i}$ and $\left(H^{i}\right)^{*}$ and $b_{i}=b_{-i}$ for all $i \in \mathbb{Z}$. That is, if $H$ satisfies Poincar\'e duality, the sequence of the Betti numbers will be symmetric, and we regard this as a numerical shadow of the duality.

If $H$ satisfies Poincar\'e duality, then we fix a such bilinear form $\langle-,-\rangle$.
In this case, we say a degree two linear map
$$\operatorname{L}: H^{i} \longrightarrow H^{i+2}$$
is a \emph{Lefschetz operator} if $\langle \operatorname{L} a, b\rangle=\langle a, \operatorname{L} b\rangle$ for all $a, b \in H$. If $\operatorname{L}$ is a Lefschetz operator, then it is said to satisfy \emph{hard Lefschetz} (HL) if for all $i \geq 0$, 
$$
\operatorname{L}^{i}: H^{-i} \longrightarrow H^{i}
$$ is an isomorphism. The following lemma is standard (see for example, Theorem 4 in Section 5, Chapter IV of \cite{serre2000complex}).

\begin{lem} \label{HL}
A degree two operator $\operatorname{L}$ on $H$ satisfies hard Lefschetz if and only if there is an action of $\mathfrak{s l}_{2}(\mathbb{R})=\langle e, f, h\rangle$ on $H$ with $\operatorname{e}$ acting as $\operatorname{L}$, and $\operatorname{h} \cdot v=i v$ for all $v \in H^{i}$. Moreover, this $\mathfrak{s l}_{2}(\mathbb{R})$-action is unique.
\end{lem}

In particular, if $\operatorname{L}$ satisfies hard Lefschetz, the Betti numbers of $H$ will satisfy  $$
\cdots \leq b_{-4} \leq b_{-2} \leq b_{0} \geq b_{2} \geq b_{4} \geq \cdots$$ 
and 
$$
\cdots \leq b_{-3} \leq b_{-1} = b_{1} \geq b_{3} \geq \cdots.
$$ 
Thus we regard the unimodality (with symmetry) as a numerical shadow of the hard Lefschetz property (assuming \emph{parity vanishing}: $
H^{\text {odd }}=0 \text { or } H^{\text {even }}=0
$).

Let $\operatorname{L}$ be a Lefschetz operator. For each $i \geq 0$, define the  \emph{Lefschetz form} on $H^{-i}$ with respect to $\operatorname{L}$ as 
$$
(a, b)_{L}^{-i}=\left\langle a, \operatorname{L}^{i} b\right\rangle
$$
for $a, b \in H^{-i}$, which is non-degenerate if and only if $\operatorname{L}$ satisfies hard Lefschetz because $\langle-,-\rangle$ is non-degenerate by assumption. 

Suppose that $\operatorname{L}$ is a Lefschetz operator satisfying hard Lefschetz. For all $i \geq 0$ set
$$
P_{L}^{-i}=\operatorname{ker}\left(L^{i+1}\right) \cap H^{-i},
$$ elements of this subspace are called \emph{primitive vectors} of degree $-i$, and this subspace is called a \emph{primitive subspace}, which is indeed the \emph{lowest weight space} in $H^{-i}$ (where we view $H$ as an $\mathfrak{s l}_{2}(\mathbb{R})$-representation via Lemma \ref{HL}.)

If $L$ is a Lefschetz operator satisfying hard Lefschetz, then we have the \emph{primitive decomposition}
$$H=\bigoplus_{i \geq 0} \bigoplus_{i \geq j \geq 0} L^{j}\left(P_{L}^{-i}\right),$$
and for $i \geq 0$,
\begin{equation} \label{pride}
H^{-i}=P_{L}^{-i} \oplus L\left(P_{L}^{-i-2}\right) \oplus L^{2}\left(P_{L}^{-i-4}\right) \oplus \cdots.
\end{equation}
 In virtue of Lemma \ref{HL}, it is the isotypic decomposition of the $\mathfrak{s l}_{2}(\mathbb{R})$-representation, where $\bigoplus_{i \geq j \geq 0} L^{j}\left(P_{L}^{-i}\right)$ is the sum of copies of the irreducible $\mathfrak{s l}_{2}(\mathbb{R})$-representations with highest weight $i$. For any Lefschetz
operator $L$, we have 
\begin{equation} \label{isome}
    (L a, L b)_{L}^{-(i-2)}=(a, b)_{L}^{-i}
    \end{equation}
for all $i \geq 2$ and $a, b \in H^{-i}$. It follows that if $L$ satisfies hard Lefschetz, then the primitive decomposition is orthogonal with respect to the Lefschetz forms.
Consequently, if $L$ satisfies hard Lefschetz, then the Lefschetz form is determined
by its restriction to the primitive subspaces.

Assume $H^{\text {odd }}=0$ or $H^{\text {even }}=0$ and that $L$ is a Lefschetz operator satisfying hard Lefschetz. We say that $(H,\langle-,-\rangle, L)$ satisfies \emph{Hodge--Riemann bilinear relations} (HR) if the restriction of the Lefschetz form to the 
primitive subspace
$$
\left.(-,-)_{L}^{m +2 i}\right|_{P_{L}^{m +2 i}}
$$
is $(-1)^{i}$-definite, where $m$ denotes the most negative integer such that $H^{m} \neq 0$. If $(H,\langle-,-\rangle, L)$ satisfies Hodge--Riemann bilinear relations, then in particular, each Lefschetz form is non-degenerate and $L$ satisfies hard Lefschetz. By \eqref{pride} and \eqref{isome}, Hodge--Riemann bilinear relations predict that the mixed signature of the Lefschetz forms is coming from some definiteness on each primitive
subspace and can be expressed in terms of the Betti numbers of $H$.

We note that in the work \cites{adiprasito2018hodge,branden2020lorentzian,huh2012milnor} mentioned in paragraph 2 of the introduction, the log-concave property is regarded as a numerical shadow of Hodge--Riemann bilinear relations by interpreting the “Betti numbers” in their context (for example, the Betti numbers of the Orlik-Solomon algebra) as some “intersection numbers”. But we think that “Betti numbers=intersection numbers” is a miracle, and one purpose of this paper is to provide a different perspective on the (equivariant) log-concave property.

The main example of graded vector spaces that satisfies this package of properties is the cohomology ring of smooth complex projective algebraic varieties whose Hodge numbers $h^{p, q}$ are zero unless $p=q$\footnote{The Hodge--Riemann bilinear relations in K\"ahler geometry for more general varieties without the assumption on the Hodge numbers are more complicated than the formulation which we use, see Section 7 of Chapter 0 of \cite{griffiths2014principles}.}, with the operator $L$ given by multiplication by an ample class. (For example, any smooth projective variety whose cohomology ring is generated by algebraic cycles will have this property.) We note that most of the K\"ahler packages discovered so far have some sort of “geometric origin". In many algebraic settings, it is often difficult to find the “lowering operator” in Lemma \ref{HL} to prove hard Lefschetz. Instead, one usually finds some notion of “positivity” or “convexity” so that the Lefschetz operator can be deformed in a family (the “K\"ahler cone” or “ample cone”), and prove the whole package by induction.

\section{Equivariant log-concavity}

When considering the equivariant log-concavity of group representations, it is natural to ask for which group $G$ and its representation $V$, the sequences of symmetric powers ${{Sym^i(V)}}$ and exterior powers ${\wedge^i(V)}$ are equivariantly log-concave? Surprisingly, the answer is for any group and any representation! Indeed, we have

\begin{prop} \label{Prop1}
For any finite-dimensional complex representation $V$ of any group $G$, the sequences of symmetric powers ${{Sym^i(V)}}$ and exterior powers ${\wedge^i(V)}$ of $V$ are equivariantly log-concave.
\end{prop}

\begin{proof}
Suppose $\dim(V)=n$, it suffices to show that the proposition holds for $G=\operatorname{GL}_{n}(\mathbb{C)}$ and $V=\mathbb{C}^{n}$, the natural representation of $G=\operatorname{GL}_{n}(\mathbb{C)}$, which follows from the classical Pieri's formula, see \cite[Proposition 15.25, Exercise 15.32 and Exercise 15.33]{fulton2013representation}.
\end{proof}

In fact, much more is true.

\begin{thm} \label{Lie}
Let $\Lambda^+$ denote the set of the dominant weights of $GL_{n}(\mathbb{C)}$, and pick any sequence $\{\lambda_{i}\}$ in $\Lambda^+$ equi-distributed on a line, then the sequence of representations $\{L(\lambda_{i})\}$ is equivariantly log-concave, where $L(\lambda_{i})$ is the finite-dimensional irreducible module with highest weight $\lambda_{i}$.
\end{thm}

Theorem \ref{Lie} is a direct corollary of the following remarkable Schur log-concavity theorem in \cite{lam2007schur}, which is a special case of Okounkov's log-concavity conjecture for the Littlewood--Richardson coefficients, see \cite{okounkov2003would}, since the Littlewood--Richardson coefficients are the structure coefficients of tensor products of finite-dimensional representations of general linear groups.

\begin{thm} \cite[Theorem 12, weak version]{lam2007schur} \label{Sch} 
For two partitions $\lambda=\left(\lambda_{1}, \lambda_{2}, \ldots, \lambda_{n}\right)$ and $\mu=\left(\mu_{1}, \mu_{2}, \ldots, \mu_{n}\right)$, suppose $\lambda+\mu$ has only even parts, and let $s_{\lambda}$ be the corresponding Schur polynomial, then
$s_{\frac{\lambda+\mu}{2}}^2 - s_{\lambda} s_{\mu}$ are Schur non-negative, which means that it is a non-negative linear combination of Schur polynomials.
\end{thm}

Theorem \ref{Lie} follows from Theorem \ref{Sch} since Schur polynomials are the characters of the corresponding highest weight irreducible representations of $\operatorname{GL}_{n}(\mathbb{C)}$\footnote{Another log-concave property of Schur polynomials is that coefficients of monomials, which are the \emph{weight multiplicities} of the Schur module, are log-concave in the root directions, see \cite{huh2022logarithmic}.}. The proof of Theorem \ref{Sch} relies on a deep result in \cite{haiman1993hecke}, which in turn relies on the Kazhdan--Lusztig conjecture for the character of simple highest weight modules of $\mathfrak{sl}_{n}(\mathbb{C})$. This is to our surprise since Theorem \ref{Lie} concerns only finite-dimensional representations. We note that giving a representation-theoretic proof and not using tools in symmetric function theory might generalize Theorem \ref{Lie} to other types.

However, giving explicit maps to realize the log-concavity in Theorem \ref{Lie} might be difficult. Even in the simplest example of $\mathfrak{g}=\mathfrak{sl}_{2}(\mathbb{C})$, by the classical Clebsch--Gordan formula, we have the explicit decomposition: 
\begin{equation*}
V(k) \otimes V(l) \cong  V(k+l) \oplus V(k+l-2) \oplus \cdots \oplus V(k-l) \cong (V(k+1) \otimes V(l-1)) \oplus V(k-l)
\end{equation*}
for $k \geq l$, where $V(k)$ is the irreducible $(k+1)$-dimensional representation. But it would be not so easy to give the explicit injections since one has to match the highest weight vectors in the tensor products. Since the exterior powers and the symmetric powers are the two fundamental constructions of representations known as Schur functors on the category of vector spaces, a natural but basic question is: can we give explicit maps to realize the log-concavity in Proposition \ref{Prop1} for $S_{n}$?  

\section{Equivariant K\"ahler packages}\label{section:eqKah}

\subsection{Polynomial rings}

Consider the polynomial ring $\mathbb{R}\left[t_{1}, \ldots, t_{n}\right]$
as a graded (we set each $t_{j}$ of degree $1$ for convenience) representation of $S_{n}$, where $S_{n}$ acts by permuting the indices. Geometrically, it is the cohomology $H^{*}(\left(\mathbb{C} \mathbb{P}^{\infty}\right)^{n}, \mathbb{R})$ of the classifying space $B (S^{1})^{n}=\left(\mathbb{C} \mathbb{P}^{\infty}\right)^{n}$ (when doubling the degrees) by the K\"unneth formula and $H^{*}\left(\mathbb{C} \mathbb{P}^{\infty} , \mathbb{R}\right)=\mathbb{R}\left[t\right]$. $S_{n}$ acts on $\left(\mathbb{C} \mathbb{P}^{\infty}\right)^{n}$ by permuting the factors hence acts on the cohomology. The Betti number $b_{2i}$ of $\left(\mathbb{C} \mathbb{P}^{\infty}\right)^{n}$
is equal to the binomial coefficient $\binom{n+i-1}{i}$, which forms a log-concave sequence for fixed $n$. We will construct an equivariant K\"ahler package in this subsection to show that $H^{*}(\left(\mathbb{C} \mathbb{P}^{\infty}\right)^{n}, \mathbb{R})$ is equivariantly log-concave for the $S_{n}$-action. Actually, it turns out that $H^{*}(\left(\mathbb{C} \mathbb{P}^{\infty}\right)^{n}, \mathbb{R})$ is \emph{strongly equivariantly log-concave} introduced in \cite{matherneequivariant}. 

Fix a pair of natural numbers $(n,m)$, consider the graded $\mathbb{R}$-vector space $$
H_{n,m}=\bigoplus_{i = 0}^{m} H_{n,m}^{-m+2i} \text{ , with }
 H_{n,m}^{-m+2i}:=D^{i} \otimes R^{m-i},$$ where $D^{i}$ denotes the degree-$i$ piece of the polynomial ring $D=\mathbb{R}\left[d_{1}, \ldots, d_{n}\right]$, $R^{m-i}$ denotes the degree-($m-i$) piece of the polynomial ring $R=\mathbb{R}\left[x_{1}, \ldots, x_{n}\right]$, $S_{n}$ acts on $D$ and $R$ by permuting the indices hence acts on $H_{n,m}$. 
 
We define a pairing on $H_{n,m}$ by setting
\begin{equation} \label{pair}
\left\langle d \otimes f, d^{\prime} \otimes f^{\prime}\right\rangle:=\left(d, f^{\prime}\right)\left(d^{\prime}, f\right)
\end{equation}
and extending linearly, where $(d, f^{\prime})=(d \cdot f^{\prime})(0) $ is the number one gets by interpreting the $d_{i}$'s as differential operators $\frac{\partial}{\partial x_{i}}$ acting on $f^{\prime}$ and evaluate $x_{i}$'s at $0$ in the result. Intuitively, one should think that this pairs “homology” with “cohomology”.

We define 
$$\operatorname{L}: H_{n,m}^{i} \longrightarrow H_{n,m}^{i+2}$$
to be the linear map 
$$
L:=\sum_{i=1}^{n} d_{i} \otimes \frac{\partial}{\partial x_{i}},
$$
where $d_{i}$ acts by multiplication.

One of our main results is that $(H_{n,m},\langle-,-\rangle, L)$ satisfies the K\"ahler package.

\begin{thm} \label{kah}
For any pair of natural numbers $(n,m)$, we have
\begin{itemize}
\item[(a)] (PD) The bilinear pairing 
$$\langle-,-\rangle: H_{n,m} \times H_{n,m} \longrightarrow \mathbb{R}$$
is an $S_{n}$-equivariant symmetric graded bilinear form, which is non-degenerate;
\item[(b)] (HL) $\operatorname{L}: H_{n,m}^{i} \longrightarrow H_{n,m}^{i+2}$
is an $S_{n}$-equivariant Lefschetz operator satisfying the hard Lefschetz theorem, i.e., for all $i \geq 0$, 
$$L^{i}: H_{n,m}^{-i} \longrightarrow H_{n,m}^{i}$$
is an isomorphism;
\item[(c)] (HR) For all $0 \leq i \leq m/2$, the bilinear form
$$(a, b)_{L}^{-m+2i}=\left\langle a, L^{m-2i} b\right\rangle: H_{n,m}^{-m+2i} \times H_{n,m}^{-m+2i} \longrightarrow \mathbb{R}$$
is $S_{n}$-equivariant and $(-1)^{i}$-definite on the primitive subspace $P_{L}^{-m+2i}=\ker\left(L^{m-2i+1}\right) \cap H^{-m+2i}$.
\end{itemize}
\end{thm}
 
\begin{proof}
 It is clear that $\langle-,-\rangle$ is a symmetric bilinear form. It is graded (i.e.,\\ $\left\langle H_{n,m}^{i}, H_{n,m}^{j}\right\rangle=0$ if $i \neq-j$) for degree reasons. It is $S_{n}$-equivariant since $\left(d, f^{\prime}\right)=\left(d \cdot f^{\prime}\right)(0)$ is $S_{n}$-equivariant. Finally, it is non-degenerate since \\
 $\left\{d^{\alpha} \otimes \frac{x^{\beta}}{\beta !} \mid 
|\alpha|=i, |\beta |=m-i\right\}$ and $\left\{d^{\beta^{\prime}} \otimes \frac{x^{\alpha^{\prime}}}{\alpha^{\prime} !} \mid |\beta^{\prime}|=m-i, |\alpha^{\prime} |=i\right\} $ form a dual\footnote{Recall the homology $H_{*}(\mathbb{C} \mathbb{P}^{\infty}, \mathbb{Z})$ of the classifying space $B S^{1}=\mathbb{C} \mathbb{P}^{\infty}$ is the \emph{divided polynomial algebras}, see \cite{hatcher2005algebraic}.} basis in $H_{n,m}^{-m+2i}$ and $H_{n,m}^{m-2i}$, where $d^{\alpha}$ (similarly for $x^{\beta}$) denotes the monomial $d_{1}^{\alpha_{1}} \cdots d_{n}^{\alpha_{n}}$ for 
$\alpha=(\alpha_{1}, \cdots ,\alpha_{n})$,  $|\alpha|:=\sum_{i=1}^{n} \alpha_{i}$ and $\beta !:=\beta_{1} ! \cdots \beta_{n} !$ for $\beta=(\beta_{1}, \cdots ,\beta_{n})$. That is,
\begin{equation} \label{dual}
\langle d^{\alpha} \otimes \frac{x^{\beta}}{\beta !},d^{\beta^{\prime}} \otimes \frac{x^{\alpha^{\prime}}}{\alpha^{\prime} !}\rangle= \begin{cases}1, & \text { if } \alpha=\alpha^{\prime}, \text { and } \beta=\beta^{\prime},\\ 0, & \text { otherwise, }\end{cases}
\end{equation}
which completes the proof of part (a).

Now we turn to the proof of part (b). We have 
$$\begin{aligned} &\left\langle L(d \otimes f), d^{\prime} \otimes f^{\prime}\right\rangle \\=&\left\langle\sum_{i=1}^{n} d_{i} d \otimes \frac{\partial}{\partial x_{i}} \cdot f, d^{\prime} \otimes f^{\prime}\right\rangle \\=& \sum_{i=1}^{n}\left(d_{i} d, f^{\prime}\right)\left(d^{\prime}, \frac{\partial}{\partial x_{i}} \cdot f\right) \\=& \sum_{i=1}^{n}\left(d, \frac{\partial}{\partial x_{i}} \cdot f^{\prime}\right)\left(d_{i} d^{\prime}, f\right) \\=&\left\langle d \otimes f, \sum_{i=1}^{n} d_{i} d^{\prime} \otimes \frac{\partial}{\partial x_{i}} \cdot f^{\prime}\right\rangle \\
=&\left\langle d \otimes f, L\left(d^{\prime} \otimes f^{\prime}\right)\right\rangle, \end{aligned}$$
which shows that $\operatorname{L}$ is a Lefschetz operator. It is clear that $\operatorname{L}$ is $S_{n}$-equivariant. To show that $\operatorname{L}$ satisfies hard Lefschetz, we define the lowering operator\\
$\operatorname{F}: H_{n,m}^{i} \longrightarrow H_{n,m}^{i-2}$ as $$\operatorname{F}:=\sum_{i=1}^{n}  \frac{\partial}{\partial d_{i}} \otimes x_{i}.$$
By direct computation,
$$\begin{aligned}
& \operatorname{L} \operatorname{F}-\operatorname{F} \operatorname{L} \\
=&\left(\sum_{i=1}^{n} d_{i} \otimes \frac{\partial}{\partial x_{i}}\right)\left(\sum_{j=1}^{n} \frac{\partial}{\partial d_{j}} \otimes x_{j}\right)-\left(\sum_{j=1}^{n} \frac{\partial}{\partial d_{j}} \otimes x_{j}\right)\left(\sum_{i=1}^{n} d_{i} \otimes \frac{\partial}{\partial x_{i}}\right) \\
=& \sum_{i=1}^{n} \sum_{j=1}^{n} (d_{i}  \frac{\partial}{\partial d_{j}} \otimes \frac{\partial}{\partial x_{i}}  x_{j}-\frac{\partial}{\partial d_{j}}  d_{i} \otimes x_{j}  \frac{\partial}{\partial x_{i}})\\
=& \sum_{i=1}^{n} (d_{i}  \frac{\partial}{\partial d_{i}} \otimes \operatorname{ id }-\operatorname{ id } \otimes x_{i}  \frac{\partial}{\partial x_{i}}),
\end{aligned}$$
the last equality follows from the Leibniz rule. 
We define 
$\operatorname{h}: H_{n,m}^{i} \longrightarrow H_{n,m}^{i}$ to be the linear map 
$$\operatorname{h}:=\sum_{i=1}^{n}\left(d_{i}  \frac{\partial}{\partial d_{i}} \otimes \operatorname{ id }-\operatorname{ id } \otimes x_{i}  \frac{\partial}{\partial x_{i}}\right).$$ Since $\sum_{i=1}^{n}d_{i}  \frac{\partial}{\partial d_{i}}$ and $\sum_{i=1}^{n}x_{i}  \frac{\partial}{\partial x_{i}}$ are Euler operators which act on homogeneous degree-$k$ polynomials with eigenvalue $k$, thus $\operatorname{h} \cdot v=i v$ for all $v \in H_{n,m}^{i}$ as required in Lemma \ref{HL}.
Finally, we check the $\mathfrak{s l_{2}}$-relations:
by definition, $\operatorname{L} \operatorname{F}-\operatorname{F} \operatorname{L}=\operatorname{h}$. We compute
\begin{equation*} 
\begin{aligned}
& \operatorname{h} \operatorname{L}-\operatorname{L} \operatorname{h} \\
=& \sum_{i=1}^{n}\left(d_{i}  \frac{\partial}{\partial d_{i}} \otimes \operatorname{ id }-\operatorname{ id } \otimes x_{i}  \frac{\partial}{\partial x_{i}}\right)\sum_{j=1}^{n} d_{j} \otimes \frac{\partial}{\partial x_{j}}-\sum_{j=1}^{n} d_{j} \otimes \frac{\partial}{\partial x_{j}}   \sum_{i=1}^{n}\left(d_{i}  \frac{\partial}{\partial d_{i}} \otimes \operatorname{ id }-\operatorname{ id } \otimes x_{i}  \frac{\partial}{\partial x_{i}}\right) \\
=& \sum_{i=1}^{n} \sum_{j=1}^{n}\left(d_{i}  \frac{\partial}{\partial d_{i}}  d_{j} \otimes \frac{\partial}{\partial x_{j}}-d_{j} \otimes x_{i}  \frac{\partial}{\partial x_{i}}  \frac{\partial}{\partial x_{j}}-d_{j}  d_{i}  \frac{\partial}{\partial d_{i}} \otimes \frac{\partial}{\partial x_{j}}+d_{j} \otimes \frac{\partial}{\partial x_{j}}  x_{i}  \frac{\partial}{\partial x_{i}}\right) \\
=& \sum_{i=1}^{n}\left((d_{i}  \frac{\partial}{\partial d_{i}}  d_{i}-d_{i}  d_{i}  \frac{\partial}{\partial d_{i}}) \otimes \frac{\partial}{\partial x_{i}}-d_{i} \otimes(x_{i}  \frac{\partial}{\partial x_{i}}  \frac{\partial}{\partial x_{i}}-\frac{\partial}{\partial x_{i}}  x_{i}  \frac{\partial}{\partial x_{i}})\right) \\
=& \sum_{i=1}^{n} 2 d_{i} \otimes \frac{\partial}{\partial x_{i}}= 2 \operatorname{L},
\end{aligned}
\end{equation*}
and similar computations show that $\operatorname{h} \operatorname{F}-\operatorname{F} \operatorname{h}=-2\operatorname{F}$. By Lemma \ref{HL}, $\operatorname{L}$ satisfies hard Lefschetz, which completes the proof of part (b).

We are now turning to the proof of part (c). The Lefschetz form $(a, b)_{L}^{-m+2i}=\left\langle a, \operatorname{L}^{m-2i} b\right\rangle$ on $H_{n,m}^{-m+2i}$ is $S_{n}$-equivariant since $\langle-,-\rangle$ and $\operatorname{L}$ are both $S_{n}$-equivariant. It remains to prove Hodge--Riemann bilinear relations, here we use a little classical Hodge theory and we refer to Section 7 of Chapter 0 of \cite{griffiths2014principles} for the reader. 
Let $V(k)$ be the irreducible $(k+1)$-dimensional representation of the Lie algebra $\mathfrak{s l}_{2}(\mathbb{R})$ with highest weight $k$.
The direct sum $\bigoplus^{n} \mathfrak{s l}_{2}(\mathbb{R})$ of $n$ copies of $\mathfrak{s l}_{2}(\mathbb{R})=\langle e_{i}, f_{i}, h_{i}\rangle$ act on
$$H_{n,m}=\bigoplus_{i = 0}^{m} H_{n,m}^{-m+2i} =\bigoplus_{i = 0}^{m} D^{i} \otimes R^{m-i}$$ via $$e_{i}=d_{i} \otimes \frac{\partial}{\partial x_{i}}, f_{i}=\frac{\partial}{\partial d_{i}} \otimes x_{i} \text{ , and } h_{i}=d_{i} \frac{\partial}{\partial d_{i}} \otimes \operatorname{ id }-\operatorname{ id } \otimes x_{i} \frac{\partial}{\partial x_{i}}.$$
Then our action of $\mathfrak{s l}_{2}(\mathbb{R})$ comes from the diagonal embedding $$\mathfrak{s l}_{2}(\mathbb{R}) \stackrel{\Delta}{\longrightarrow} \oplus^{n}{\mathfrak{s l}_{2}(\mathbb{R})},$$ where $$\Delta(L)=\sum_{i=1}^{n} e_{i}, \Delta\left(F\right)=\sum_{i=1}^{n} f_{i} \text{ , and } \Delta(h)=\sum_{i=1}^{n} h_{i}.$$
As a representation of $\bigoplus^{n} \mathfrak{s l}_{2}(\mathbb{R})$, or equivalently, as a representation of the enveloping algebra $U(\bigoplus^{n} \mathfrak{s l}_{2}(\mathbb{R})) \cong \bigotimes^{n} U\left(\mathfrak{s l}_{2}(\mathbb{R})\right)$, we have
$$H_{n,m} \cong \bigoplus_{\alpha=(\alpha_{1}, \cdots, \alpha_{n}), |\alpha|=m} V\left(\alpha_{1}\right) \otimes \cdots \otimes V\left(\alpha_{n}\right),$$
where each $V\left(\alpha_{i}\right)$ is an irreducible representation of the $i$-th copy of $\mathfrak{s l}_{2}(\mathbb{R})$ in $\bigoplus^{n} \mathfrak{s l}_{2}(\mathbb{R})$ and the other copies of $\mathfrak{s l}_{2}(\mathbb{R})$ in $\bigoplus^{n} \mathfrak{s l}_{2}(\mathbb{R})$ act trivially. 
Therefore, as a representation of $\mathfrak{s l}_{2}(\mathbb{R})=\langle L, F, h\rangle$, we have the decomposition into irreducible representations
\begin{equation} \label{dec}
\begin{aligned}
H_{n,m} & \cong \bigoplus_{\alpha=(\alpha_{1}, \cdots, \alpha_{n}), |\alpha|=m} V\left(\alpha_{1}\right) \otimes \cdots \otimes V\left(\alpha_{n}\right) \\
& \cong \bigoplus_{\alpha=(\alpha_{1}, \cdots, \alpha_{n}), |\alpha|=m} H^{*}\left(\mathbb{C} \mathbb{P}^{\alpha_{1}}, \mathbb{R}\right) \otimes \cdots \otimes H^{*}\left(\mathbb{C} \mathbb{P}^{\alpha_{n}},\mathbb{R})\right.\\
& \cong \bigoplus_{\alpha=(\alpha_{1}, \cdots, \alpha_{n}), |\alpha|=m} H^{*}\left(\mathbb{C} \mathbb{P}^{\alpha_{1}} \times \cdots \times \mathbb{C} \mathbb{P}^{\alpha_{n}}, \mathbb{R}\right),
\end{aligned}
\end{equation}
where the third isomorphism is by the K\"unneth formula, the second is via the action 
$$\operatorname{f_{i}}\left(t_{i}^{j}\right)=(\alpha_{i}-j+1) t_{i}^{j-1},
\operatorname{h_{i}}\left(t_{i}^{j}\right)=(2 j-\alpha_{i}) t_{i}^{j} \text{ , and } \operatorname{e_{i}}\left(t_{i}^{j}\right)=(j+1) t_{i}^{j+1}$$ 
on the cohomology of projective space $H^{*}\left(\mathbb{C} \mathbb{P}^{\alpha_{i}} , \mathbb{R}\right)=\mathbb{R}\left[t_{i}\right]/\langle t_{i}^{\alpha_{i}+1}\rangle$. We fix an isomophism such that the highest weight vector $d^{\alpha} \otimes 1 \in H_{n,m}$ of $\bigoplus^{n} \mathfrak{s l}_{2}(\mathbb{R})$ corresponds to $t^{\alpha} \in H^{*}\left(\mathbb{C} \mathbb{P}^{\alpha_{1}} \times \cdots \times \mathbb{C} \mathbb{P}^{\alpha_{n}}, \mathbb{R}\right)$. By the actions of $\operatorname{f_{i}}'s$, the lowest weight vector $1 \otimes x^{\alpha} \in H_{n,m}$ corresponds to $1 \in H^{*}\left(\mathbb{C} \mathbb{P}^{\alpha_{1}} \times \cdots \times \mathbb{C} \mathbb{P}^{\alpha_{n}}, \mathbb{R}\right)$. Therefore, the bilinear form $\langle-,-\rangle$ on $H_{n,m}$ restricts to each $H^{*}\left(\mathbb{C} \mathbb{P}^{\alpha_{1}} \times \cdots \times \mathbb{C} \mathbb{P}^{\alpha_{n}}, \mathbb{R}\right)$ as the usual Poincar\'e pairing up to a positive scalar by \eqref{dual}. 
Under this isomophism, each $e_{i}$ corresponds to multiplying an ample class on
$H^{*}\left(\mathbb{C P}^{\alpha_{i}}, \mathbb{R}\right)$. Therefore, $\operatorname{L}$ also corresponds to multiplying an ample class on $ H^{*}\left(\mathbb{C} \mathbb{P}^{\alpha_{1}} \times \cdots \times \mathbb{C} \mathbb{P}^{\alpha_{n}}, \mathbb{R}\right)$. By classical Hodge theory, Hodge--Riemann bilinear relations hold for each $H^{*}\left(\mathbb{C P}^{\alpha_{1}} \times \cdots \times \mathbb{C P}^{\alpha_{n}}, \mathbb{R}\right)$. This give Hodge--Riemann bilinear relations for $H_{n,m}$ since it is not hard to see that the decomposition \eqref{dec} is orthogonal with respect to the Lefschetz form in $H_{n,m}$ by \eqref{dual} again, which completes the proof of part (b). The proof is completed.
\end{proof}

The hard Lefschetz property immediately gives the following corollary.
\begin{cor}
$H^{*}(\left(\mathbb{C} \mathbb{P}^{\infty}\right)^{n}, \mathbb{R})=\mathbb{R}\left[t_{1}, \ldots, t_{n}\right]$ is strongly equivariantly log-concave for any $m \geq 0$:
\begin{equation} \label{str}
V^{0} \otimes V^{m} \subset V^{1} \otimes V^{m-1} \subset V^{2} \otimes V^{m-2} \subset \cdots \subset \begin{cases}V^{m / 2} \otimes V^{m / 2} & \text { if } m \text { is even } \\ V^{(m-1) / 2} \otimes V^{(m+1) / 2} & \text { if } m \text { is odd, }\end{cases}
\end{equation}
where $V^{i}$ denotes the degree-$2i$ piece of $H^{*}(\left(\mathbb{C} \mathbb{P}^{\infty}\right)^{n}, \mathbb{R})=\mathbb{R}\left[t_{1}, \ldots, t_{n}\right]$ (here each $t_{i}$ is of degree-$2$). Furthermore, the inclusions of $S_{n}$-representations are given by the operator 
$$
\operatorname{L}=\sum_{i=1}^{n} t_{i} \otimes \frac{\partial}{\partial t_{i}},$$
where $t_{i}$ acts by multiplication.
\end{cor}

\begin{rmk} \label{wider}
Actually, the Lefschetz operator $\operatorname{L}$ is indeed $\operatorname{GL}_{n}(\mathbb{R)}$-equivariant, see for example, \cite{xi2012module}. In particular, $\operatorname{L}$ is equivariant for the symmetric powers of any representation of any group $G$. But the whole packages, in particular, the bilinear form $\langle-,-\rangle$ and the Lefschetz form $
(a, b)_{L}^{-m+2i}$ are only $S_{n}$-equivariant, or more generally, for which group $G$ acts on $\mathbb{R}\left[t_{1}, \ldots, t_{n}\right] = {{Sym^i(V)}}$ with $V \cong V^{*}$. A similar remark holds for the following exterior algebra case.
\end{rmk}

\begin{rmk} 
Each primitive subspace measures the difference of the adjacent tensor products of representations in \eqref{str}. The Lefschetz forms $
(a, b)_{L}^{-m+2i}$ are equivariant and Hodge--Riemann bilinear
relations tell us these equivariant Lefschetz forms are positive definite or negative definite on these differences of adjacent tensor products of representations, depending only on the graded degree. A similar remark holds for the following exterior algebra case.
\end{rmk}

\subsection{Exterior algebras}

Consider the exterior algebra $\Lambda_{\mathbb{R}}\left[\alpha_{1}, \ldots, \alpha_{n}\right]$ as a graded (each $\alpha_{j}$ is of degree 1) representation of $S_{n}$, where $S_{n}$ acts by permuting the indices hence acts on exterior powers. Geometrically, it is the cohomology $H^{*}\left((S^{1})^{n}, \mathbb{R}\right)$ of the torus $T=(S^{1})^{n}$ by the K\"unneth formula and $H^{*}\left(S^{1}, \mathbb{R}\right) =\mathbb{R}\left[\alpha\right]/(\alpha^{2})$. $S_{n}$ acts on the torus $T=(S^{1})^{n}$ by permuting the factors hence acts on the cohomology. The Betti number $b_{i}$ of the torus $T=(S^{1})^{n}$
is equal to the binomial coefficient $\binom{n}{i}$, which forms a classical log-concave sequence for fixed $n$. This time, we will construct an equivariant K\"ahler package to show that $H^{*}(\left(S^{1}\right)^{n}, \mathbb{R})$ is strongly equivariantly log-concave for the $S_{n}$ action.

Fix a pair of natural numbers $(n,m)$ with $m \leq 2n$, consider the graded $\mathbb{R}$-vector space
$$
H_{n,m}^{\prime}=\bigoplus_{i=0}^{m} (H_{n,m}^{\prime})^{-m+2 i} \text { , with } (H_{n,m}^{\prime})^{-m+2 i}:=\Lambda^{i} \otimes (\Lambda^{*})^{m-i},
$$
where $\Lambda^{i}$ denotes the degree- $i$ piece of the exterior algebra $\Lambda=\Lambda(V)=\Lambda_{\mathbb{R}}\left[\theta_{1}, \ldots, \theta_{n}\right],$ $(\Lambda^{*})^{m-i}$ denotes the degree- $(m-i)$ piece of the exterior algebra $\Lambda^{*}=\Lambda(V^{*})=\Lambda_{\mathbb{R}}\left[\xi_{1}, \ldots, \xi_{n}\right]$, note that $\Lambda^{i}=0$ (resp., $(\Lambda^{*})^{m-i}=0$) except $0 \leq i \leq n$ (resp., $0 \leq m-i \leq n$. We demand that $\theta_{1}, \ldots, \theta_{n}$ and $\xi_{1}, \ldots, \xi_{n}$ form dual basis for $V$ and $V^{*}$. $S_{n}$ acts on $\Lambda$ and $\Lambda^{*}$ as the exterior powers of the corresponding permutation representations hence acts on $H_{n,m}^{\prime}$.

We define a pairing \footnote{Again, one should think intuitively that this pairs “homology” with “cohomology”!} on $H_{n,m}^{\prime}$ by setting
\begin{equation} \label{pair2}
\left\langle u \otimes v^{*}, u^{\prime} \otimes\left(v^{*}\right)^{\prime}\right\rangle:= (-1)^\varepsilon \left(u,\left(v^{*}\right)^{\prime}\right)\left(u^{\prime}, v^{*}\right)
\end{equation}
and extending linearly, where $(-,-)$ is the bilinear map of $\Lambda^{i} \times (\Lambda^{*})^{j} \rightarrow \mathbb{R}$ which on indecomposable elements $u=u_{1} \wedge \cdots \wedge u_{i}$ in $\Lambda^{i}$ and $v^{*}=v_{1}^{*} \wedge \cdots \wedge v_{j}^{*}$ in $(\Lambda^{*})^{j}$ yields
\begin{equation} \label{det}
\left(u, v^{*}\right)=
\begin{cases} \operatorname{det}\left(v_{k}^{*}\left(u_{l}\right)\right), & \text { if } i=j, \\ 
0, & \text { if } i \neq j.\end{cases}
\end{equation}
 and
\begin{equation} \label{corr}
\varepsilon= \begin{cases}m(m-1)/2, & \text { if } 0 \leq m \leq n, \\ (2n-m)(2n-m-1)/2, & \text { if } n < m \leq 2n.\end{cases}
\end{equation}
Here the global sign correction $(-1)^\varepsilon$ is to make Hodge--Riemann bilinear
relations satisfied with the \emph{standard sign}, that is, the later-defined Lefschetz form is positive
definite on the lowest non-zero degree of $H_{n,m}^{\prime}$ (which is necessarily primitive
because the later-defined operator satisfies the hard Lefschetz theorem).

Recall that for each $\alpha \in V^{*}$ (or $v \in V$, respectively), it is possible to define an anti-derivation $i_{\alpha}$ called the \emph{interior product} with $\alpha$ on the algebra $\Lambda(V)$ (or $\Lambda(V^{*})$, resp., we only consider the first case in the remaining of this paragraph):
$$i_{\alpha}: \Lambda^{k} V \rightarrow \Lambda^{k-1} V,$$
which satisfies $\left(i_{\alpha} \mathbf{w}\right)\left(v_{1}^{*}, v_{2}^{*}, \ldots, v_{k-1}^{*}\right)=\mathbf{w}\left(\alpha, v_{1}^{*}, v_{2}^{*}, \ldots, v_{k-1}^{*}\right)$ for $\mathbf{w} \in \Lambda^{k} V$ and $v_{1}^{*}, v_{2}^{*}, \ldots, v_{k-1}^{*}$ are $k-1$ elements of $V^{*}$. It follows that if $v$ is an element of $V\left(=\Lambda^{1} V\right)$, then $i_{\alpha} v=\alpha(v)$ is the dual pairing between elements of $V$ and elements of $V^{*}$. Furthermore, $i_{\alpha}$ is a graded derivation of degree $-1$:
$$i_{\alpha}(u \wedge v)=\left(i_{\alpha} u\right) \wedge v+(-1)^{\operatorname{deg} u} u \wedge\left(i_{\alpha} v\right) .
$$ Additionally, let $i_{\alpha} f=0$ whenever $f$ is a pure scalar (i.e., belonging to $\Lambda^{0} V$ ).

For each $v \in V$ (similarly for $\alpha \in V^{*}$), let $e_{v}$ denote the \emph{exterior (wedge) product} with $v$ on the left:
$$e_{v}: \Lambda^{k-1} V \rightarrow \Lambda^{k} V.$$

The following lemma is standard and we omit the proof.
\begin{lem} \label{exc}
For any $u, v \in V$ and $\alpha, \beta \in V^{*}$, we have
\begin{itemize}
\item[(a)]$i_{\alpha} \circ i_{\beta}=-i_{\beta} \circ i_{\alpha}$, in particular, $i_{\alpha} \circ i_{\alpha}=0$;
\item[(b)]$e_{v} \circ e_{u}=-e_{u} \circ e_{v}$, in particular, $e_{v} \circ e_{v}=0$;
\item[(c)]$i_{\alpha} \circ e_{v}+ e_{v} \circ i_{\alpha}=\alpha(v) \cdot \operatorname{ id }$, $i_{v} \circ e_{\alpha}+ e_{\alpha} \circ i_{v}=\alpha(v) \cdot \operatorname{ id }$;
\item[(d)](Adjoint) The bilinear pairing $(-,-)$ satisfies $\left(i_{\alpha}(u), \beta\right)=\left(u, e_{\alpha} \beta\right)=\left(u, \alpha \wedge \beta\right)$. 
\end{itemize}
\end{lem}

We define 
$$L: (H_{n,m}^{\prime})^{i} \longrightarrow (H_{n,m}^{\prime})^{i+2}$$
to be the linear map 
$$L:=\sum_{k=1}^{n} e_{\theta_{k}} \otimes i_{\theta_{k}}.$$

Our second main result is that $(H_{n,m}^{\prime},\langle-,-\rangle, L)$ also satisfies the K\"ahler package.

\begin{thm} \label{kah2} 
For any pair of natural numbers $(n,m)$ with $m \leq 2n$, we have
\begin{itemize}
\item[(a)] (PD)  The bilinear pairing 
$$\langle-,-\rangle: H_{n,m}^{\prime} \times H_{n,m}^{\prime} \longrightarrow \mathbb{R}$$
is an $S_{n}$-equivariant symmetric graded bilinear form, which is non-degenerate;
\item[(b)] (HL) $L: (H_{n,m}^{\prime})^{i} \longrightarrow (H_{n,m}^{\prime})^{i+2}$
is an $S_{n}$-equivariant Lefschetz operator satisfying the hard Lefschetz theorem, i.e., for all $i \geq 0$, 
$$L^{i}: (H_{n,m}^{\prime})^{-i} \longrightarrow (H_{n,m}^{\prime})^{i}$$
is an isomorphism;
\item[(c)] (HR) For $0 \leq i \leq m/2$, the bilinear form 
$$(a, b)_{L}^{-m+2i}=\left\langle a, L^{m-2i} b\right\rangle: (H_{n,m}^{\prime})^{-m+2i} \times (H_{n,m}^{\prime})^{-m+2i} \longrightarrow \mathbb{R}$$
is $S_{n}$-equivariant and $(-1)^{i}$-definite on the primitive subspace $(P_{L}^{\prime})^{-m+2i}=\operatorname{ker}\left(L^{m-2i+1}\right) \cap (H_{n,m}^{\prime})^{-m+2i}$.
\end{itemize}
\end{thm}

\begin{proof}
The proof is very similar to that of Theorem \ref{kah}.
It is clear that $\langle-,-\rangle$ is a symmetric bilinear form. It is graded by the definition of $(-,-)$ and it is $S_{n}$-equivariant since $(-,-)$ is  $S_{n}$-equivariant. Finally, it is non-degenerate since $\left\{\theta^{\wedge \alpha} \otimes \xi^{\wedge \beta} \mid 
|\alpha|=i, |\beta |=m-i\right\}$ and $\left\{\theta^{\wedge \beta^{\prime}} \otimes {\xi^{\wedge \alpha^{\prime}}} \mid
|\beta^{\prime}|=m-i, |\alpha^{\prime} |=i\right\} $ form a dual basis in $(H_{n,m}^{\prime})^{-m+2i}$ and $(H_{n,m}^{\prime})^{m-2i}$, where $\theta^{\wedge \alpha}$ (similarly for $\xi^{\wedge \beta}$) denotes the exterior power $\theta_{1}^{\alpha_{1}} \wedge \cdots \wedge \theta_{n}^{\alpha_{n}}$ for 
$\alpha=(\alpha_{1}, \cdots ,\alpha_{n})$ where $\alpha_{j}=0 \text{ or } 1$, $|\alpha|:=\sum_{i=1}^{n} \alpha_{i}$. That is, $$
\langle \theta^{\wedge \alpha} \otimes \xi^{\wedge \beta},\theta^{\wedge \beta^{\prime}} \otimes {y^{\wedge \alpha^{\prime}}}\rangle
\begin{cases}1, & \text { if } \alpha=\alpha^{\prime}, \text { and } \beta=\beta^{\prime},\\ 0, & \text { otherwise }\end{cases},
$$
which completes the proof of part (a).

Let us turn to the proof of part (b). The adjoint property (d) of Lemma \ref{exc}
shows that $L$ is a Lefschetz operator. It is clear that $L$ is $S_{n}$-equivariant. To show that $L$ satisfies hard Lefschetz, we define the lowering operator
$F: H_{n,m}^{i} \longrightarrow H_{n,m}^{i-2}$ as 
$$F:=\sum_{k=1}^{n} i_{\xi_{k}} \otimes e_{\xi_{k}}.$$

A direct computation based on Lemma \ref{exc} shows that
$$L F-F L=\sum_{k=1}^{n} (e_{\theta_{k}} \circ i_{\xi_{k}} \otimes \operatorname{ id }- \operatorname{ id } \otimes e_{\xi_{k}} \circ i_{\theta_{k}}).$$
 If we define 
$h: (H_{n,m}^{\prime})^{i} \longrightarrow (H_{n,m}^{\prime})^{i}$ to be the linear map 
$$h:=\sum_{k=1}^{n} (e_{\theta_{k}} \circ i_{\xi_{k}} \otimes \operatorname{ id }- \operatorname{ id } \otimes e_{\xi_{k}} \circ i_{\theta_{k}}),$$ 
then it is not hard to see that $(\sum_{k=1}^{n}e_{\theta_{k}} \circ i_{\xi_{k}})(\theta^{\wedge \alpha} )=
|\alpha| \theta^{\wedge \alpha}$ and similarly for $\sum_{k=1}^{n}\left(e_{\xi_{k}} \circ i_{\theta_{k}}\right)$, thus $h \cdot v=i v$ for all $v \in (H_{n,m}^{\prime})^{i}$ as required.

Finally, one can check the $\mathfrak{s l_{2}}$-relations for $\langle L, F, h\rangle$ using Lemma \ref{exc}. By Lemma \ref{HL}, $L$ satisfies hard Lefschetz.

The proof of part (c), especially the proof of Hodge--Riemann bilinear relations, is quite similar to that of part (c) in Theorem \ref{kah} given earlier except that the action of the triple $\mathfrak{s l}_{2}(\mathbb{R})=\langle e_{k}, f_{k}, h_{k}\rangle$ should be replaced with $$e_{k}=e_{\theta_{k}} \otimes i_{\theta_{k}}, f_{k}=i_{\xi_{k}} \otimes e_{\xi_{k}} \text{ , and } h_{k}=e_{\theta_{k}} \circ i_{\xi_{k}} \otimes \operatorname{ id }- \operatorname{ id } \otimes e_{\xi_{k}} \circ i_{\theta_{k}},$$ and each irreducible representation $V\left(\alpha_{i}\right)$ is at most $2$-dimensional. It is easy to check that the global sign correction $\varepsilon$ in \eqref{corr} makes the bilinear form $\langle-,-\rangle$ on $H_{n,m}^{\prime}$ restricts to the usual Poincar\'e pairing on each cohomology of product of projective space up to a positive scalar in a decomposition similar to \eqref{dec}, we leave the details to the reader. The proof is completed.
 \end{proof}

We have the following corollary of the hard Lefschetz theorem.

\begin{cor}
$H^{*}((S^{1})^{n}, \mathbb{R})=\Lambda_{\mathbb{R}}\left[\alpha_{1}, \ldots, \alpha_{n}\right]$ is strongly equivariantly log-concave for any $0 \leq m \leq 2n$: $$V^{0} \otimes V^{m} \subset V^{1} \otimes V^{m-1} \subset V^{2} \otimes V^{m-2} \subset \cdots \subset \begin{cases}V^{m / 2} \otimes V^{m / 2} & \text { if } m \text { is even } \\ V^{(m-1) / 2} \otimes V^{(m+1) / 2} & \text { if } m \text { is odd, }\end{cases}$$
where $V^{i}$ denotes the degree-$i$ piece of $H^{*}((S^{1})^{n}, \mathbb{R})=\Lambda_{\mathbb{R}}\left[\alpha_{1}, \ldots, \alpha_{n}\right]$, which vanishes except $0 \leq i \leq n$. Furthermore, the inclusions of $S_{n}$-representations are given by the operator $$
L=\sum_{k=1}^{n} e_{\alpha_{k}} \otimes i_{\alpha_{k}^{*}}.$$
where $e_{\alpha_{k}}$ and $i_{\alpha_{k}^{*}}$ are the exterior product and interior product, respectively, and $\alpha_{k}^{*}$ is the linear functional on $V= \langle \alpha_{1}, \ldots, \alpha_{n} \rangle$ satisfying $\alpha_{k}^{*}\left(\alpha_{l}\right)=\delta_{k l}= \begin{cases}0, & \text { if } k \neq l, \\ 1, & \text { if } k=l.\end{cases}$
\end{cor}

\begin{rmk}
In \cite[Proposition 5.7]{gedeon2017equivariant}, the authors gave an alternative proof of the equivariant log-concavity  for the exterior powers using formulas for the Kronecker products of Schur functions of hook
shapes without giving the inclusion maps.
\end{rmk}

\begin{rmk}
For our application to equivariant log-concavity, it is only necessary to establish the equivariant hard Lefschetz property, here we quote the following in \cite{braden2020singular}
\begin{quote}
“The one insight that we can take away is that, while the hard Lefschetz theorem is
typically the main statement needed for applications, it is always necessary to prove Poincar\'e duality, the hard Lefschetz theorem, and the Hodge--Riemann relations together as a single package.”
\end{quote}
\end{rmk}

\subsection{An interesting example}

For the exterior algebra, we now show that “usual” gradings on tensor products satisfy Poincar\'e duality and hard Lefschetz but not Hodge--Riemann bilinear relations.

Fix any natural numbers $n$, consider the graded $\mathbb{R}$-vector space
$$H_{n}=\bigoplus_{i=0}^{2n} H_{n}^{-n+i} \text { , with } H_{n}^{-n+i}:=\bigoplus_{j+k=i} \Lambda^{j} \otimes (\Lambda^{*})^{k},$$
where $\Lambda^{j}$ denotes the degree- $j$ piece of the exterior algebra $\Lambda=\Lambda(V)=\Lambda_{\mathbb{R}}\left[\theta_{1}, \ldots, \theta_{n}\right],$ $(\Lambda^{*})^{k}$ denotes the degree- $k$ piece of the exterior algebra $\Lambda^{*}=\Lambda(V^{*})=\Lambda_{\mathbb{R}}\left[\xi_{1}, \ldots, \xi_{n}\right]$, i.e., we consider the “usual” gradings on the tensor product $\Lambda \otimes \Lambda^{*}$ up to a degree shift of $-n$. Again, we demand that $\theta_{1}, \ldots, \theta_{n}$ and $\xi_{1}, \ldots, \xi_{n}$ form dual basis of $V$ and $V^{*}$. $S_{n}$ acts on $\Lambda$ and $\Lambda^{*}$ as the exterior powers of the corresponding permutation representations hence acts on $H_{n}$.

We define a graded pairing $\langle-,-\rangle$ on $H_{n}$ by the multiplication map
\begin{equation}
H_{n}^{-n+i} \otimes H_{n}^{n-i} \rightarrow H_{n}^{n} \cong \mathbb{R},
\end{equation}
where the linear isomorphism maps $\theta_{1} \wedge \ldots \wedge \theta_{n} \otimes \xi_{1} \wedge \ldots \wedge \xi_{n}$ to $1$.
Note the unusual grading convention here, since the identity
element $1 \otimes 1$ lives in degree $-n$.

We define 
$$\operatorname{L}: H_{n}^{i} \longrightarrow H_{n}^{i+2}$$
to be the linear map 
$$\operatorname{L}:=\sum_{k=1}^{n} \operatorname{e_{\theta_{k}}} \otimes \operatorname{e_{\xi_{k}}}.$$

We have the following theorem.

\begin{thm} \label{kah3} 
For any natural numbers $n$, we have
\begin{itemize}
\item[(a)] (PD)
$$\langle-,-\rangle: H_{n} \times H_{n} \longrightarrow \mathbb{R}$$
is an $S_{n}$-equivariant graded bilinear form, which is non-degenerate;
\item[(b)] (HL) $\operatorname{L}: H_{n}^{i} \longrightarrow H_{n}^{i+2}$
is $S_{n}$-equivariant and satisfies the hard Lefschetz theorem, i.e., for all $i \geq 0$, 
$$
\operatorname{L}^{i}: H_{n}^{-i} \longrightarrow H_{n}^{i}
$$
is an isomorphism.
\end{itemize}
\end{thm}

\begin{proof}
By definition, $\langle-,-\rangle$ is a graded bilinear form. It is $S_{n}$-equivariant and non-degenerate since the multiplication map
\begin{equation} \label{multi}
\Lambda^{j} \otimes \left(\Lambda^{*}\right)^{k} \bigotimes \Lambda^{n-j} \otimes \left(\Lambda^{*}\right)^{n-k} \rightarrow H_{n}^{n} \cong \mathbb{R}
\end{equation} 
is an $S_{n}$-equivariant perfect pairing for any $0 \leq j, k \leq n$,
which proves part (a). 

Let us turn to the proof of part (b). It is clear that $\operatorname{L}$ is $S_{n}$-equivariant. To show that $\operatorname{L}$ satisfies hard Lefschetz, we define the lowering operator
$\operatorname{F}: H_{n}^{i} \longrightarrow H_{n}^{i-2}$ as 
$$
\operatorname{F}:=\sum_{k=1}^{n} \operatorname{i}_{\xi_{k}} \otimes \operatorname{i}_{\theta_{k}}.$$

A direct computation based on Lemma \ref{exc} shows that
$$
 \operatorname{L} \operatorname{F}-\operatorname{F} \operatorname{L}=\sum_{k=1}^{n} (\operatorname{e}_{\theta_{k}} \circ \operatorname{i}_{\xi_{k}} \otimes \operatorname{ id }+ \operatorname{ id } \otimes \operatorname{e}_{\xi_{k}} \circ \operatorname{i}_{\theta_{k}})-n \operatorname{ id } \otimes \operatorname{ id }.
 $$
 If we define 
$\operatorname{h}: H_{n}^{i} \longrightarrow H_{n}^{i}$ to be the linear map 
$$\operatorname{h}:=\sum_{k=1}^{n} (\operatorname{e}_{\theta_{k}} \circ \operatorname{i}_{\xi_{k}} \otimes \operatorname{ id }+ \operatorname{ id } \otimes \operatorname{e}_{\xi_{k}} \circ \operatorname{i}_{\theta_{k}})-n \operatorname{ id } \otimes \operatorname{ id },
$$ 
then it is not hard to see that $\operatorname{h} \cdot v=i v$ for all $v \in H_{n}^{i}$ as required. By checking the $\mathfrak{s l_{2}}$-relations for $\langle \operatorname{L}, \operatorname{F}, \operatorname{h} \rangle$ using Lemma \ref{exc}, $\operatorname{L}$ satisfies hard Lefschetz by Lemma \ref{HL}.
The proof is completed.
\end{proof}

\begin{rmk} \label{coun}
The main difference between $H_{n}$ and $H_{n,m}$ (or $H_{n,m}^{\prime}$) is that $H_{n}$ does not satisfy the parity vanishing. Note that commuting the multiplication order of the factors in \eqref{multi} yields a sign of $(-1)^{j(n-j)+k(n-k)}$, thus when $n$ is even, the bilinear form $\langle-,-\rangle$
is skew-symmetric on the odd part.
By part (b) of
Lemma \ref{exc}, $\operatorname{L}$ is a Lefschetz (i.e., self-adjoint) operator with respect to the symmetric bilinear form $\langle-,-\rangle$
except $n$ is even and $j+k=i$ is odd, and is skew-self-adjoint with respect to the skew-symmetric bilinear form $\langle-,-\rangle$ if it is the case.
So in any case, the Lefschetz form is a symmetric bilinear form. However, the Hodge--Riemann bilinear relations are NOT satisfied. For example, when $n=2$, it is easy to compute the signature of the Lefschetz form is $(2,2)$ on degree $-1$ and is $(3,3)$ on degree $0$, neither of which matches the prediction of Hodge--Riemann bilinear relations. 
\end{rmk}

\begin{rmk}
There is also an $S_{n}$-equivariant $\mathfrak{s l_{2}}$-action on the “usual” gradings of the tensor product of the polynomial ring $D=\mathbb{R}\left[d_{1}, \ldots, d_{n}\right]$ and $R=\mathbb{R}\left[x_{1}, \ldots, x_{n}\right]$: $$H_{n}^{\prime}=\bigoplus_{i=0}^{\infty} (H_{n}^{\prime})^{i} \text { , with } (H_{n}^{\prime})^{i}:=\bigoplus_{j+k=i} D^{j} \otimes R^{k},$$
via 
$$\operatorname{e}=-\sum_{l=1}^{n} \frac{\partial}{\partial d_{l}} \otimes \frac{\partial}{\partial x_{l}}, 
\operatorname{f}=\sum_{l=1}^{n} d_{l} \otimes x_{l} \text{ , and } 
\operatorname{h}_{i}=-n \operatorname{ id } \otimes \operatorname{ id }-\sum_{l=1}^{n} (d_{l} \frac{\partial}{\partial d_{l}} \otimes \operatorname{ id }+\operatorname{ id } \otimes x_{l} \frac{\partial}{\partial x_{l}}).$$
But neither Poincar\'e duality nor hard Lefschetz holds and the representation $H_{n}^{\prime}$ is infinite-dimensional.
\end{rmk}

\section{Questions and Future work}

We would like to finish with some questions and conjectures. Since our constructions of these K\"ahler packages are purely algebraic, it is natural to ask the following question first.

\begin{ques}
Are there any geometric interpretations of these constructions?
\end{ques}

In \cite{kim2022lefschetz}, the authors use the equivariant hard Lefschetz in Theorem \ref{kah3} (which is their Theorem 3.2) as a key tool to study the fermionic diagonal coinvariants. We would like to ask the following

\begin{ques}
Do the other two equivariant K\"ahler packages (Theorem \ref{kah} and Theorem \ref{kah2}), especially the Hodge--Riemann relations, have some implications for the diagonal coinvariant ring and the fermionic diagonal coinvariant ring?
\end{ques}

In \cite{novak2020increasing}, the authors transformed the log-concavity conjecture of the distribution of permutations in $S_{n}$ with fixed length of longest increasing subsequences in \cite{chen2008log} into the following equivariant log-concavity conjecture of $S_{n}$

\begin{conj} \cite[Conjecture 2]{novak2020increasing}
Let $\lambda$ be a partition of $n$ and $V^{\lambda}$ be the irreducible representation of $S_{n}$ corresponding to $\lambda$, then the graded representation $V_{n}=\bigoplus_{k=0}^{n} V_n^{k}$ of $S_{n}$ is equivariantly log-concave, where
$$V_{n}^{k}=\bigoplus_{\substack{\lambda \vdash n \\ \ell(\lambda)=k}} f^{\lambda} V^{\lambda},$$
here $f^{\lambda}$ is the dimension of $V^{\lambda}$, also equal to the number of standard Young tableaux of shape $\lambda$, which is given by the hook-length formula.
\end{conj}

Based on the representation stability theory and some numerical evidence, the first author has the following conjecture in \cite{gui2022on}

\begin{conj} \label{flag}
For all integer $n \geq 1$, the cohomology ring of the flag manifold $U_{n} / T$, which is also known as the coinvariant ring of $S_{n}$: $H^{2 *}(U_{n} / T, \mathbb{R})\cong \mathbb{R}\left[t_{1}, \ldots, t_{n}\right] /\left(\sigma_{1}, \ldots, \sigma_{n}\right)$, is equivariantly log-concave as graded representation of $S_{n}$, where each $t_{j}'s$ is of degree 2 and $\sigma_{i}$ is the $i$-th elementary symmetric polynomials in the variables $t_1, \ldots, t_n$.
\end{conj}

It is interesting to see whether similar equivariant Hodge-theoretic structures exist in the settings of the above two equivalent log-concavity conjectures and the settings of equivalent log-concavity conjectures in \cite{matherneequivariant}.

\textbf{Acknowledgment}. 

The first author would like to thank Ming Fang, Peter L. Guo, Hongsheng Hu, Ruizhi Huang, and Nanhua Xi for the useful discussions, and thank Nicholas Proudfoot and Brendon Rhoades for helpful comments, and is grateful to Neil J.Y. Fan for the help with computer calculations. This work was completed when the second author visited the AMSS, CAS. The second author would like to thank Sian Nie for the invitation. The authors would like to thank the anonymous referee for the comments and suggestions. The second author was supported in part by the National Natural Science Foundation of China (No. 11922119). 

\bibliography{template}

\end{document}